\documentclass[11pt]{amsart}
\usepackage{mathtools,amssymb,latexsym,amsmath,mathdots,mathrsfs,stix}
\usepackage[colorlinks=true, linkcolor=black, citecolor=blue]{hyperref}
\usepackage{bookmark,bbm}
\usepackage{enumitem}
\usepackage{xcolor}
\input{xypic}
\xyoption{all}
\numberwithin{equation}{section}

\newcommand{\Z}{{\mathbb Z}}
\newcommand{\C}{{\mathbb C}}

\newcommand{\G}{{\mathbf G}}
\newcommand{\B}{{\rm B}}
\newcommand{\T}{{\rm T}}

\newcommand{\J}{{\rm J}}

\renewcommand{\o}{{\mathfrak o}}

\newcommand{\W}{{\mathcal W}}

\newcommand{\p}{\mathfrak{p}}
\newcommand{\GL}{{\rm GL}}

\newcommand{\U}{{\rm U}}

\newtheorem*{acknowledgment}{Acknowledgment}

\numberwithin{equation}{section}
\newtheorem{theorem}[equation]{Theorem}

\newtheorem{lemma}[equation]{Lemma}
\newtheorem{prop}[equation]{Proposition}

\newtheorem{defn}[equation]{Definition}
\newtheorem{remark}[equation]{Remark}

\begin{document}

\title{Stability of Asai local factors for $GL(2)$}
\date{\today}
\thanks{}
\author{\bf Yeongseong Jo and M. Krishnamurthy }
\address{Yeongseong Jo, Department of Mathematics, University of Iowa, Iowa City, IA 52242, USA.} \email{yeongseong-jo@uiowa.edu}

\address{M. Krishnamurthy, Department of Mathematics, University of Iowa, Iowa City, IA 52242, USA.}
\email{muthu-krishnamurthy@uiowa.edu}

\begin{abstract} 
Let $F$ be a non-archimedean local field of characteristic not equal to $2$ and let $E/F$ be a quadratic algebra. We prove the stability of local factors attached to (complex) irreducible admissible representations of $GL(2,E)$ via the Rankin-Selberg method under highly ramified twists. This includes both the Asai as well as the Rankin-Selberg local factors attached to pairs. Our method relies on expressing the gamma factor as a Mellin transform using Bessel functions. \end{abstract}

\maketitle

\tableofcontents

\section{\bf Introduction and statement of the main result}
\label{sec:intro}
Let $\G=\GL_2$ and let $F$ be a non-archimedean local field whose ${\rm char}(F)\neq 2$. Let $q$ denote  the cardinality of the residue field. Let $E$ be a quadratic algebra over $F$. We consider the local $\gamma$-factor attached to an irreducible admissible (complex) representation $\Pi$ of $\G(E\otimes_F F)$ by the Rankin-Selberg method \cite{Fli1,JPSS, Kab}. Let $\gamma(s,\Pi,\chi,\psi)$ denote this $\gamma$-factor, where $\chi$ is a character (smooth) of $F^{\times}$ and $\psi=\psi_F$ is a fixed non-trivial character of $F$. Then 
\[
\gamma(s,\Pi,\chi,\psi)=\left\{\begin{array}{cc}\gamma_{\rm As}(s,\Pi,\chi,\psi)&\mbox{ if }E\mbox{ is a field}\\\gamma(s,(\pi\otimes\chi)\times\sigma,\psi)&\mbox{ if }E=F\times F.\end{array}\right.
\]
Here $\gamma(s,(\pi\otimes\chi)\times \sigma,\psi)$ is the Rankin-Selberg $\gamma$-factor attached to the representation $\Pi=(\pi\otimes\chi)\times\sigma$ of $\GL_2(F)\times \GL_2(F)$ \cite{JPSS}, and $\gamma_{\rm As}(s,\Pi,\chi,\psi)$ is the {\it  Asai} $\gamma$-factor studied by Flicker \cite{Fli1}. When $E$ is a field, if $\chi_E$ denotes a character of $E^{\times}$ whose restriction to $F^{\times}$ is $\chi$, then this is in fact the Asai $\gamma$-factor attached to the representation $\Pi\otimes\chi_E$ of $\G(E)$. By definition, it is independent of the choice of $\chi_E$. In this paper, we prove 
\begin{theorem}\label{main}
Suppose $\Pi_1$ and $\Pi_2$ are irreducible admissible representations of $\GL_2(E\otimes F)$ having the same central character. Then there exists an integer $N$ so that 
\[
\gamma(s,\Pi_1,\chi,\psi)=\gamma(s,\Pi_2,\chi,\psi)
\]
for all characters $\chi$ whose conductor $n(\chi)>N$. In other words, the local $\gamma$-factor $\gamma(s,\Pi,\chi,\psi)$ is stable for $\chi$ sufficiently highly ramified. In this situation, 
\[
  L(s,\Pi_1,\chi)=L(s,\Pi_2,\chi)\equiv 1,
\]
and consequently the epsilon factor stabilize as well: 
\[
\varepsilon(s,\Pi_1,\chi,\psi)=\varepsilon(s,\Pi_2,\chi,\psi).
\]
\end{theorem}

In the split case, the above theorem is not new, it is in fact proved by Jacquet and Shalika \cite{JS} for $\GL_n\times \GL_n$, $n\geq 2$. They first prove the unequal case, i.e, $\GL_n\times \GL_m, n\neq m$, and then deduce the equal case by embedding $\GL_n\times \GL_n\subset \GL_n\times \GL_{n+1}$, in order to avoid using the more complicated Rankin-Selberg integral in the equal case. On the other hand, our proof (for $\GL_2\times \GL_2$) follows from a direct manipulation of the underlying local integral. Also, when ${\rm char}(F)=0$ and $E$ is a field, the above theorem follows from \cite[Theorem A]{CCI} since the stability property for Artin $\gamma$-factors is known \cite{DH}. The proof there is global-to-local where as the proof here, using the theory of Howe vectors \cite{H}, is purely local. To the best of our knowledge, Theorem~\ref{main}  is new in the positive characteristic case. More importantly, we believe the approach here will be applicable to other situations where stability remains an open probem. (See Remark below.) We would also like to point out the recent works  \cite{AKMSS, BP, shank} on comparing the different notions of Asai local factors using global-to-local methods.

\par
We now outline our proof of Theorem~\ref{main} which is inspired by the works of Baruch \cite{Bar2} and Chai and Zhang \cite{ChaiZhang}. For an irreducible admissible representation $\Pi$ of $\GL_2(E\otimes F)$, recall the local zeta integral 
\[
 Z(s,W,\Phi,\chi)=\int_{\U(F) \backslash \G(F)} W(g)\Phi(e_2g) \chi(\mathrm{det}(g)) |\mathrm{det}(g)|^s_F\; dg.
\]
Here, $W$ lies in the Whittaker model of $\Pi$, and $\Phi$ is a Schwartz-Bruhat function on $F^2$. Also, we have $\widetilde{W}$, the Whittaker function dual to $W$, which belongs to the Whittaker model of the contragredient representation $\widetilde{\Pi}$. If $\widehat{\Phi}$ denotes the Fourier transform of $\Phi$ relative to $\psi$, then there is the associated dual zeta integral $Z(1-s,\widetilde{W},\widehat{\Phi},\chi^{-1})$. The gamma factor $\gamma(s,\Pi,\chi,\psi)$ is a rational function of $q^{-s}$ satisfying the relation
\[
  Z(1-s,\widetilde{W},\widehat{\Phi},\chi^{-1})=\gamma(s,\Pi,\chi,\psi)Z(s,W,\Phi,\chi)
\]
for all pairs $(W,\Phi)$.

\par
For $m\gg 0$ suitably large and an appropriate $\Phi$, we  show that $Z(s,W_m,\Phi,\chi)$ is a non-zero constant on $s$ and that it only depends on $\pi$ through its central character. This is done in Proposition~\ref{lhs}. Here, $W_m$ is the Howe Whittaker function as  in (\ref{howe-whitt}) on the group $\GL_2(E\otimes F)$. Following \cite{Zhang2}, we use the notation $B_m(g,W)$ to denote $W_m(g)$, since it is also a partial Bessel function in the sense of \cite{CShT}. For $\Pi_1$ and $\Pi_2$ as in Theorem~\ref{main}, let $\omega$ denote the restriction of the common central character to $F^{\times}$. Taking Whittaker functions $W_{\Pi_1}, W_{\Pi_2}$ in their Whittaker models, respectively, let $(W_{\Pi_1})_m=W_{\Pi_1,m}$ and $(W_{\Pi_2})_m=W_{\Pi_2,m}$ denote the corresponding Howe Whittaker functions. First, we show that the the difference $Z(1-s,\widetilde{W}_{\Pi_1,m},\widehat{\Phi},\chi^{-1})-Z(1-s,\widetilde{W}_{\Pi_2,m},\widehat{\Phi},\chi^{-1})$ (see \eqref{compactform}) is
\[
\approx \int\limits_{F^{\times}} [B_m(t(a,1)w,W_{\Pi_1})-B_m(t(a,1)w,W_{\Pi_2})] (\omega\chi)^{-1}(a) |a|_F^{-s}
   d^{\times}a,
  \]
where the implied constant depends only on $\omega,\chi,\psi$ and $m$. We exploit properties of Howe Whittaker functions to arrive at this equality, particularly Proposition~\ref{Big Cell-GL(2)}, which controls the support of the function $B_m(\cdot, W_{\Pi_1})-B_m(\cdot,W_{\Pi_2})$ away from the identity.

\par
We note that $m$ in the above expression depends on the conductor $n(\chi)$ of $\chi$. At this point, we split the proof into two cases according to whether $E$ is a field or not. Using properties of the functions $W_m$ and the full Bessel function, we show that the function $a\mapsto  B_m(t(a,1)w,W_{\Pi_1})-B_m(t(a,1)w,W_{\Pi_2})$ is in fact independent of $m$ for sufficiently large $m$. This in turn implies that the above integral is zero for a suitably highly highly ramified $\chi$, thus concluding the proof of our Theorem. (The claim about the $L$-function may be proved along the lines of \cite[Proposition 5.1]{JS}.) This is all explained in Section~\ref{sec-GL(2)}. We also summarize the required properties of $W_m$ in Subsection~\ref{sec-howe}.

\par
\begin{remark}We hope to generalize our method to prove the stability of the Rankin-Selberg Asai, ${\rm Sym}^2$ and $\wedge^2$ gamma factors  for $\GL_n$. This is non-trivial -- the main obstacle being the analogue of Proposition~\ref{Big Cell-GL(2)} for $\GL_n$ $($see \cite[Lemma 6.2.2]{Bar1}$)$ which introduces other relevant Weyl elements, besides the identity and the long element, while analyzing the partial Bessel function $B_m(\cdot, W_1-W_2)$. We speculate the techniques of \cite{CShT} to be useful in this regard. 
\end{remark}

\section{\bf Preliminaries} 
\subsection{Asai local factors}
\label{sec-asai}
Let $\G=\GL_2$. For any local field $F$, we write $\o_F$ to denote its ring of integers and ${\mathfrak p}_F$ for the maximal ideal in $\mathfrak{o}_F$. Let ${\rm val}_{F}: F^{\times}\rightarrow \Z$ be the associated discrete valuation map. We fix a uniformizer  $\varpi_F$, so that ${\p}_F=(\varpi_F)$ and ${\rm val}_{F}(\varpi_F)=1$. The corresponding absolute value $|x|_F=q^{-{\rm val}_F(x)}$ is the normalized absolute of $F$ with $|\varpi_F|_F=q^{-1}$, where $q=|\o_F/{\p}_F|$ is the cardinality of the residue field. For the algebraic group $\G=\GL_2$, let $\B=\T\U$ denote the Borel subgroup of upper triangular matrices, where
\[
  \T=\left\{ t(a,b):=\begin{pmatrix} a& \\ & b \end{pmatrix}\; \middle| \; a,b \in {\mathbb G}_m  \right\}
\]
is the maximal torus consisting of diagonal matrices and 
\[
  \U=\left\{ u(x):= \begin{pmatrix} 1 & x \\ & 1 \end{pmatrix}  \; \middle| \; x \in {\mathbb G}_a \right\} 
\]
is the unipotent radical of $\B$. Let ${\rm Z}$ denote the center of $\GL_2$, and let ${\rm A}$ denote the subtorus 
\[
{\rm A}=\left\{t(a,1)\;\middle|\; a\in {\mathbb G}_m\right\}. 
\]
Let 
\[
  \overline{\U}=\left\{ \overline{u}(x):= \begin{pmatrix} 1 &  \\ x & 1 \end{pmatrix}  \; \middle| \; x \in {\mathbb G}_a \right\} 
\]
be the unipotent subgroup opposed to $\U$. We write 
\[
w=\begin{pmatrix} &1 \\ 1& \end{pmatrix}
\] 
 to denote the long Weyl element in $\GL_2$. Recall the Bruhat decomposition 
\[
   \G(F)=\B(F) \cup  \B(F) w\U(F), 
\]
with uniqueness of expression, i.e., every $g\not\in \B(F)$ has a unique expression of the form  $g=bwu, b\in \B(F), u\in \U(F)$. Put ${\rm K}={\rm K}_F=\GL_2(\o_F)$, the standard maximal compact subgroup of $\G(F)$.

\par
 Let $F$ be a non-archimedean local field whose ${\rm char}(F)\neq 2$. Let $E$ be a quadratic algebra over $F$ with the associated trace map $\text{tr}_{E/F}$. There are two possibilities for $E: (\mathrm{1})$ $E=F \times F$ and $(2)$ $E$ is a field. We fix an additive character $\psi=\psi_F$ of $F$. In Case (1), it is convenient to set $\xi=(1,-1)$, and in Case (2), we fix an element $\xi \in E-F$ such that ${\rm tr}_{E/F}(\xi)=0$. We define the non-trivial additive character $\psi_E$ of $E$ by $\psi_E(x)=\psi({\rm tr}_{E/F}(\xi x))$, then $\psi_E$ is trivial on $F$. (Note that in the split case $\psi_E=(\psi,\psi^{-1})$.) The character $\psi_E$ defines a non-degenerate character of $\U(E)$ via $u(x)\mapsto \psi_E(x)$. Similarly, $\psi$ defines one for $\U(F)$. 
 
 \par
Suppose $\pi$ is an irreducible admissible representation of $\GL_2(F)$. If $\pi$ is {\it generic}, we write $\W(\pi,\psi)$ to denote associated Whittaker model relative to $\psi$. If $\pi$ is not generic, then $\pi$ is of the form $\pi=\eta_F\circ{\det}$ for some character $\eta_F$ of $F^{\times}$ and may be realized as the {\it Langlands quotient} of the (normalized) induced representation $(\rho,I(\eta_F|\cdot|_F^{1/2},\eta_F|\cdot|_F^{-1/2}))$. We fix a non-zero Whittaker functional $\lambda$ on this space and form
\[
 \mathcal{W}(\rho,\psi)=\{ W_v(g)=\lambda(\rho(g)v) \; | \; v \in I(\eta_F|\cdot|_F^{1/2},\eta_F|\cdot|_F^{-1/2}), g \in \GL_2(F) \}
\]
with the natural action of $\GL_2(F)$ by right translation. This will then play the role of the Whittaker model for $\pi$ and by abuse of notation we set $\mathcal{W}(\pi,\psi):=\W(\rho,\psi)$. Let $\pi^{\iota}$ denote the representation of $\GL_2(F)$ with $\pi^{\iota}(g)=\pi({^tg^{-1}})$ on the same underlying space $V$. It is well-known that if $\pi$ is irreducible, $\pi^{\iota}$ is isomorphic to its contragredient representation $\widetilde{\pi}$. If $\widetilde{W}(g):=W(w{^tg^{-1}})$, then $\widetilde{W}$ lies in $\mathcal{W}(\pi^{\iota},\psi^{-1})$. 

Suppose $\Pi$ is an irreducible admissible representation of $\GL_2(E \otimes F)$ with associated Whittaker model $\W(\Pi,\psi_E)$, that is,
\begin{itemize}
	\item Case $(1)$ $\Pi=\pi\otimes \sigma$, where $\pi$ (resp. $\sigma$) is an irreducible admissible representations of $\GL_2(F)$, and $\mathcal{W}(\Pi,\psi_E)=\mathcal{W}(\pi,\psi) \otimes \mathcal{W}(\sigma,\psi^{-1})$;
	\item Case $(2)$ $\Pi$ is an irreducible admissible representation of $\GL_2(E)$ with the associated Whittaker model $\mathcal{W}(\Pi,\psi_E)$.
\end{itemize}
 In Case (1), we note that $\W(\Pi,\psi_E)$ is generated as a linear space by functions $W$ of the form 
 \[
 W(g_1,g_2)=W_1(g_1)W_2(g_2), W_1\in \W(\pi,\psi), W_2\in \W(\sigma,\psi^{-1}), g_1,g_2\in \GL_2(F).
 \]
 In this situation, we write $W(g)=W(g,g)$ to denote its value on the diagonal subgroup $\GL_2(F)\hookrightarrow \GL_2(E\otimes F)$.
 \par
 Let $C_c^{\infty}(F^2)$ be the space of locally constant, compactly supported functions $\Phi : F^2 \rightarrow \mathbb{C}$. We write $e_1=(1,0)$ and $e_2=(0,1)$ to denote the standard basis for $F^2$. The Fourier transform of $\Phi \in  C_c^{\infty}(F^2)$ with respect to $\psi$ is given by
\[
 \widehat{\Phi}(y)= \widehat{\Phi}_{\psi}(y)=\int_{F^2} \Phi(x) \psi(x{^ty}) \; dx,
\]
where $dx$ is the self-dual measure for which the Fourier inversion formula takes the form $\widehat{\widehat{\Phi}}(x)=\Phi(-x)$.
 
 \par
For $W \in \mathcal{W}(\Pi,\psi_E)$, $\Phi \in C_c^{\infty}(F^2)$, a smooth character $\chi$ of $F^{\times}$, and $s\in \C$, consider the zeta integral \cite{Fli1, Kab}
 \[
   Z(s,W,\Phi,\chi)=\int_{\U(F) \backslash \G(F)} W(g)\Phi(e_2g) \chi(\mathrm{det}(g)) |\mathrm{det}(g)|^s_F\; dg,
 \]
 where $\G(F)\hookrightarrow \G(E)$ embedded diagonally. 
In Case (2), if $\chi_E$ is a character of $E^{\times}$ whose restriction to $F^{\times}$ is $\chi$, then this is the {\it Flicker} integral \cite{Fli1} attached to the representation $\pi\otimes\chi_E$ and it clearly only depends on the restriction of $\chi$. 

\par
For each $(s,W,\Phi)$ as above, the integral $Z(s,W,\Phi,\chi)$ converges absolutely when $\Re(s)$ is sufficiently large and defines a rational function of $q^{-s}$. The collection of all such integrals span a $\C[q^s,q^{-s}]$-fractional ideal of $\C(q^{-s})$. Further, there is a function $\gamma(s,\Pi,\chi,\psi,\xi)\in \C(q^{-s})$ satisfying the local functional equation
\begin{equation}\label{lfe}
 Z(1-s,\widetilde{W},\widehat{\Phi},\chi^{-1})=\gamma(s,\Pi,\chi,\psi,\xi)Z(s,W,\Phi,\chi) 
\end{equation}
We refer the reader to \cite[Theorem 2.7]{JPSS} and \cite[Theorem 2]{Kab} for proofs of these properties. We write 
\[
\gamma(s,\Pi,\chi,\psi,\xi)=\left\{\begin{array}{cc}\gamma(s,(\pi\otimes\chi)\times\sigma,\psi)&\mbox{ in Case (1)}\\\gamma_{\rm As}(s,\Pi,\chi,\psi,\xi)&\mbox{ in Case (2)}\end{array}\right.
\]
to denote the Rankin-Selberg $\gamma$-factor and the {\it Asai} $\gamma$-factor, respectively.  From loc.cit., we also have the corresponding Asai $L$-factor $L_{\rm As}(s,\Pi,\chi)$ and the Asai $\varepsilon$-factor $\varepsilon_{\rm As}(s,\Pi,\chi,\psi,\xi)$ satisfying  
\[
\varepsilon_{\rm As}(s,\Pi,\chi,\psi,\xi)=\gamma_{\rm As}(s,\Pi,\chi,\psi,\xi)\frac{L_{\rm As}(s,\Pi,\chi)}{L_{\rm As}(1-s,\Pi^{\iota},\chi^{-1})},
\]
and  the Rankin-Selberg $L$-factor $L(s,(\pi\otimes\chi)\times\sigma)$ and $\varepsilon$-factor $\varepsilon(s,(\pi \otimes \chi) \times \sigma,\psi)$ satisfying 
\[
  \varepsilon(s,(\pi \otimes \chi) \times \sigma,\psi)=\gamma(s,(\pi \otimes \chi) \times \sigma,\psi) \frac{L(s,(\pi \otimes \chi) \times \sigma)}{L(1-s,(\pi^{\iota}\otimes \chi^{-1}) \times \sigma^{\iota})}.
\]

\subsection{Dependence on the pair $(\psi,\xi)$}
We end this Section by examining the dependence of the local $\gamma$-factor on the pair $(\psi,\xi)$. A different choice of $(\psi',\xi')$ results in elements $a,b\in F^{\times}$, satisfying $\psi'(x)=\psi(ax), x\in F$, and $\xi'=b\xi$. Let $\omega_{\Pi}$ denote the central character of $\Pi$ and put $\omega=\omega_{\Pi}|_{F^{\times}}$. In Case $(1)$, $\omega=\omega_{\pi}\omega_{\sigma}$; and in Case $(2)$, $\omega=\omega_{\Pi}|_{F^{\times}}$. We have the following well-known fact whose proof we include here for the sake of completeness.
\begin{lemma}
\label{unr}
For $a, b \in F^{\times}$ as above, we have 
\begin{eqnarray}
 \gamma(s,(\pi\otimes\chi) \times \sigma,\chi,\psi')&=&\omega^2(a)\chi^4(a)|a|^{4s-2}\gamma(s,(\pi\otimes\chi) \times \sigma,\chi,\psi) \nonumber \\
 \gamma_{\rm As}(s,\Pi,\chi,\psi',\xi)&=&\omega^2(a)\chi^4(a)|a|^{4s-2}\gamma_{\rm As}(s,\Pi,\chi,\psi,\xi) \nonumber
 \end{eqnarray}

and
\[
  {  \gamma_{\rm As}(s,\Pi,\chi,\psi,\xi')=\omega(b)\chi^{2}(b)|b|^{2s-1}\gamma_{\rm As}(s,\Pi,\chi,\psi,\xi). }
\]
\end{lemma}
\begin{proof}
We give the proof in the non-split case. A similar proof works in the split case as well. Let $W'(g)=W(t(a,1)g)$, then one checks that $W'\in {\mathcal W}(\Pi,\psi'_E)$ and the map $W\mapsto W'$ is a bijection from ${\mathcal W}(\Pi,\psi_E)$ to ${\mathcal W}(\Pi,\psi'_E)$. It is easy to see that 
\[
\begin{split}
    Z(s,W',\Phi,\chi)&=\int_{\U(F) \backslash \G(F)} W'(g)\Phi(e_2g) \chi(\mathrm{det}(g)) |\mathrm{det}(g)|^s_F\; dg\\
    &=\chi^{-1}(a)|a|_F^{1-s}Z(s,W,\Phi,\chi).
 \end{split}   
\]
For $\Phi \in C_c^{\infty}(F^2)$, if $\widehat{\Phi}_{\psi'}$ denotes the Fourier transform with respect to $\psi'$, then 
\[
\widehat{\Phi}_{\psi'}(x,y)=|a|_F\widehat{\Phi}_{\psi}(ax,ay).
\]
Thus (with $\widehat{\Phi}=\widehat{\Phi}_{\psi}$)
\[
\begin{split}
  &Z(1-s,\widetilde{W'},\widehat{\Phi}_{\psi'})\\
  &=\omega(a)|a|_F\int_{\U(F) \backslash \G(F)} \widetilde{W}(t(a,1)g)\widehat{\Phi}(e_2t(a,a)g) \chi^{-1}(\mathrm{det}(g)) |\mathrm{det}(g)|^{1-s}_F\; dg\\
  &=\omega^2(a)\chi^2(a)|a|_F^{2s-1} \int_{\U(F) \backslash \G(F)} \widetilde{W}(t(a,1)g)\widehat{\Phi}(e_2g) \chi^{-1}(\mathrm{det}(g)) |\mathrm{det}(g)|^{1-s}_F\; dg\\
  &=\omega^2(a)\chi^3(a)|a|_F^{3s-1}Z(1-s,\widetilde{W},\widehat{\Phi}).\\
 \end{split}   
\]
The first assertion now follows from the local functional equation (\ref{lfe}). A similar calculation works for the second assertion. One only needs to observe that, changing $\xi\mapsto \xi'$ but keeping $\psi$ unchanged, alters the additive character with respect to which the Whittaker model is considered but has no effect on the Fourier transform. 
\end{proof}

Henceforth, we fix a $\psi=\psi_F$ so that it is unramified, and also fix a choice of $\xi$ in Case (2) so that the corresponding $\psi_E$ is also unramified and we suppress $\xi$ in the subsequent notation. Thus we write $\gamma(s,\Pi,\chi,\psi)$ to denote 
\[
\gamma(s,\Pi,\chi,\psi)=\left\{\begin{array}{cc}\gamma(s,(\pi\otimes\chi)\times\sigma,\psi)&\mbox{ in Case (1)}\\\gamma_{\rm As}(s,\Pi,\chi,\psi)&\mbox{ in Case (2).}\end{array}\right.
\]

\subsection{Howe vectors and partial Bessel functions} 
\label{sec-howe}

 We review the theory of Howe vectors \cite{H} for $\GL_2$ over a local field $F$ which was subsequently studied by Baruch \cite{Bar1} in a more general context. Let $q$ denote  the cardinality of the residue field. Let $\psi$ be an additive character of $F$ of conductor $0$, i.e., $\psi$ is trivial on $\mathfrak{o}_F$ while $\psi|_{{\p}^{-1}_F}\neq 1$. Suppose $(\pi,V)$ is an irreducible admissible representation of $\GL_2(F)$ with the associated Whittaker model ${\mathcal W}=\W(\pi,\psi)$ relative to $\psi$.

 \par
 For $m\geq 1$, let ${\rm K}_m$ be the $m$-th congruence subgroup, i.e., ${\rm K}_m={\rm I}+{\rm M}_2(\p_F^m)$. On ${\rm K}_m$, define the function $\tau_m$ by the formula
\[
\tau_{m}(k)=\psi(\varpi_F^{-2m}k_{1,2}), 
\]
where $(k_{i,j})$ is the matrix of $k$. One checks that $\tau_m$ is a linear (unitary) character of ${\rm K}_m$, trivial on ${\rm K}_{2m}$. 
Put 
\[
d_m=\begin{pmatrix}\varpi_F^{-2m}&\\&{1}\end{pmatrix},
\]
and let $\J_m=d_m{\rm K}_md_m^{-1}$. Then $\J_m$ is given by
\[
   \J_m=\begin{pmatrix} 1+{\p}_F^m          & {\p}_F^{-m}       \\
                                      {\p}_F^{3m}       & 1+{\p}_F^m        \\
                                    \end{pmatrix}.
\]
Define a character $\varphi_m$ of $\J_m$ by $\varphi_m(j)=\tau_m(d_m^{-1}jd_m)$, $j\in \J_m$.

\par
It is well-known (cf. \cite{BuHe}) that ${\rm K}_m$ is {\it decomposed} with respect to $\B$, i.e., the product map 
\[
{\rm K}_m\cap {\U}(F)\times {\rm K}_m\cap {\T}(F)\times {\rm K}_m\cap {\overline{\U}}(F)\longrightarrow {\rm K}_m
\]
is a bijection, in fact, a homeomorphism of topological spaces. Since $d_m$ is diagonal, it follows that the group ${\rm J}_m$ is also decomposed with respect to $\B$. Further, 
\[
d_m({\rm K}_m\cap {\U}(F))d_m^{-1}\supset {\rm K}_m\cap {\U}(F) \quad \text{and} \quad d_m({\rm K}_m\cap {\overline{\U}}(F))d_m^{-1}\subset {\rm K}_m\cap {\overline{\U}}(F),
\]
i.e., conjugation by $d_m$ enlarges the {\it upper} part ${\rm K}_m\cap {\U}(F)$ of ${\rm K}_m$ while shrinking its {\it lower} part ${\rm K}_m\cap {\overline{\U}}(F)$. 
Let us put $\J_m^{\mathscr u}={\J}_m\cap \U(F)$ and $\J^{\mathscr l}_m={\J}_m\cap \overline{\U}(F)$. One checks that $\varphi_m$ and $\psi$ both agree on ${\rm J}^{\mathscr u}_m$.

\par
For $m\geq 1$, consider the function $f_m$ in the Hecke algebra of $\G(F)$ given by 
\[
f_m=\left\{\begin{array}{cc}
\frac{1}{\rm{vol}(\J_m)}\varphi_m^{-1}&\text{ on }\J_m\\ 0&\text{ outside of }\J_m\end{array}\right..
\]
For $v$ in $V$, put $v_m=\lambda(f_m)v$. Explicitly, $v_m=\frac{1}{\rm{vol}(\J_m)}\int\limits_{\J_m}\varphi_m^{-1}(j)(\lambda(j)v)dj$. 
\begin{defn}
Fix a $v$ satisfying $\lambda(v)=1$ and let $l$ be an integer so that ${\rm K}_{l}$ fixes $v$. Then the vector $v_m$ is called a Howe vector of $\pi$ if $m\geq l$.  
\end{defn}
An important property of Howe vector $v_m$ is that it is also given by the formula 
\[
v_m=\frac{1}{\rm{vol}(\J^{\mathscr u}_m)}\int\limits_{\J^{\mathscr u}_m}(\lambda(u)v)\psi^{-1}(u)du.
\]
To see this, one has to utilize the decomposition $\J_m=\J_m^{\mathscr u}\cdot ({\J_m}\cap\overline{\B}(F))$ and observe that $\J_m\cap \overline{\B}(F)\subset {\rm K}_m$ fixes $v$ for $m\geq l$. For $W\in {\mathcal W}$, we define $W_m\in {\mathcal W}$ by 
\begin{equation}\label{howe-whitt}
W_m(g)=\frac{1}{\text{vol}(\J^{\mathscr u}_m)}\int\limits_{\J^{\mathscr u}_m}W(gu)\psi^{-1}(u)du.
\end{equation}
Note, if $W=W_v$ and $v_m$ is a Howe vector, then $W_m=W_{v_m}$ is the corresponding Howe Whittaker function. We collect  the important properties of the functions $W_m$ in the following Lemma. (cf. \cite[Lemma 5.2]{Bar2}).
\begin {lemma}
\label{Howe}
Choose $W\in {\mathcal W}$ so that $W({\rm I})=1$. Let $l$ be such that $\rho({\rm K}_l)W=W$, where $\rho$ denotes right translation. Then we have
\begin{enumerate}[label=$(\arabic*)$]
\item\label{Howe-item1} $W_m({\rm I})=1$;
\item\label{Howe-item2} If $m \geq l$, then $W_m(gj)=\varphi_m(j)W_m(g)$ for all $j \in \J_m$;
\item\label{Howe-item3} If $m \geq k$, then
\[
  W_m(g)=\frac{1}{\mathrm{vol}(\J^{\mathscr u}_m)} \int_{\J^{\mathscr u}_m}  W_k(gu) \psi^{-1}(u) du.
\]
\end{enumerate}
\end{lemma}

We also have the following Lemma (cf. \cite[Lemma 3.2]{ChaiZhang}) concerning the support of $W_m$ on the diagonal torus.
\begin{lemma}
\label{a-support}
Let $m\geq l$ and $W_m$ be as in Lemma~\ref{Howe}. Then we have
\begin{enumerate}[label=$(\arabic*)$]
\item\label{a-support-item1} $W_m(t(a,1))\neq 0 \iff a \in 1+{\p}^{m}_F$.
\item\label{a-support-item2} If $W_m(t(a,1)w) \neq 0$, then $a \in {\p}^{-3m}_F$.
\end{enumerate}

\end{lemma}

\begin{proof}
For $x \in \mathfrak{p}_F^{-m}$ we have $u(x) \in \J^{\mathscr u}_m$. From the relation
\[
  \begin{pmatrix} a & \\ & 1 \end{pmatrix}\begin{pmatrix} 1 & x \\ & 1 \end{pmatrix}=\begin{pmatrix} 1 & ax \\ & 1 \end{pmatrix} \begin{pmatrix} a &  \\ & 1 \end{pmatrix},
\]
we obtain $\psi(x) W_m(t(a,1))=\psi(ax)W_m(t(a,1))$. If $a \notin 1+\mathfrak{p}_F^m$, we have $(1-a)x \notin \o_F \subset \mathrm{ker}(\psi)$ for some $x \in \mathfrak{p}_F^{-m}$. Thus we get that $W_m(t(a,1))=0$. If $a \in 1+\mathfrak{p}_F^m$, then $t(a,1) \in  \J^{\mathscr u}_m$. Our result follows from Lemma \ref{Howe} that $W_m(t(a,1))=\varphi_m(t(a,1))W_m({\rm I})=\varphi_m(0)=1$.

\par
To see property \ref{a-support-item2}, for $x \in \mathfrak{p}_F^{3m}$, it is easy to see that $\overline{u}(x) \in \J^{\mathscr l}_m$. Using
\[
 \begin{pmatrix} a & \\ & 1 \end{pmatrix} w \begin{pmatrix} 1 &  \\ x & 1 \end{pmatrix}=\begin{pmatrix} 1 & ax \\ & 1 \end{pmatrix} \begin{pmatrix} a &  \\ & 1 \end{pmatrix} w,
\]
we observe that $W_m(t(a,1)w)=\psi(ax)W_m(t(a,1)w)$. The conclusion follows from the argument similar to that in property \ref{a-support-item1}.
\end{proof}

By virtue of Lemma \ref{Howe}, $W_m$ is what one calls a {\it partial Bessel function}, in the sense that, 
\[
  W_m(u_1gu_2)=\psi(u_1)\psi(u_2)W_m(g), u_1 \in \U(F), u_2 \in \J_m^{\mathscr u}, g \in \GL_2(F). 
\]
We use the notation $B_m(g,W)=W_m(g)$, $g \in \G(F)$. The main result that we need regarding these partial Bessel functions is \cite[Lemma 3.6]{ChaiZhang} (slightly paraphrased here):
\begin{prop}
\label{Big Cell-GL(2)}
Let $\pi$ and $\sigma$ be irreducible admissible representation of $\GL_2(F)$ with the same central character. Let $l$ be an integer so that $W_1 \in \mathcal{W}(\pi,\psi)$ and $W_2 \in \mathcal{W}(\sigma,\psi)$ are Whittaker functions fixed by ${\rm K}_l$. Assume $W_{i}({\rm I})=1, i=1,2$. Then, for $m \geq 3l$, the function given by the difference 
\[
  B_m(g,W_1)-B_m(g,W_2)
\]
is supported on $g \in \B(F) w \J^{\mathscr u}_m$. 
\end{prop} 

We also need the full Bessel function for the representation $\pi$. Let $\{ \J^{\mathscr u}_m \}_{m \geq 0}$ be an exhaustive filtration by compact open subgroup of $\U(F)$. If $W_v \in \mathcal{W}(\pi,\psi)$ then the integral
\[
  \ell(v,x)=\int\limits_{\J_m^{\mathscr u}} W_v \left(  \begin{pmatrix} x &  \\  & 1 \end{pmatrix} w \begin{pmatrix} 1 & u  \\  & 1 \end{pmatrix}   \right) \psi^{-1}(u)  du
\]
converges in the sense that it stabilizes for large $m$ depending on $x$. (See Lemma 4.1of \cite{soudry}.) It defines a Whittaker functional on $V$ for fixed $x \in F$, as does the functional $v \mapsto W_v(e)$. From the uniqueness of Whittaker functionals, it follows that there is the constant proportionality, as a function of $x$ when $\ell(v,x)$ varies, that is,
\[
  j_{\pi}(x) W_v(e)=\int\limits_{\J_m^{\mathscr u}} W_v \left(  \begin{pmatrix} x &  \\  & 1 \end{pmatrix} w \begin{pmatrix} 1 & u  \\  & 1 \end{pmatrix}   \right) \psi^{-1}(u)  du.
\]
The function $j_{\pi}(x)$ is called {\it the Bessel function} (attached to $w$) and its basic properties were studied by Soudry \cite[Lemma 4.2]{soudry} which we recall below. The partial Bessel function introduced above (which is not quasi-invariant on the right under the full unipotent subgroup) is related to the full Bessel function.

\begin{prop} Let $l$ and $W \in \mathcal{W}$ be as in Lemma~\ref{Howe}. 
\label{soudry}
\begin{enumerate}[label=$(\arabic*)$]
\item\label{soudry-item2} For $n\geq \frac{3}{2}l$, we have 
\[
  \mathrm{vol}(\J_{3n}^{\mathscr u}) W_{3n} (t(x,1)w)= j_{\pi}(x) \quad {\text for\; all} \; |x| \leq q^{6n}.
\]
\item\label{soudry-item3}  For any non-negative integer $n$, $x \in \mathfrak{p}_F^{-6n}$ and $a \in 1+\mathfrak{p}_F^{3n}$, we have
\[
 j_{\pi}(xa)=j_{\pi}(x).
\]
\end{enumerate}
\end{prop}

\section{\bf Proof of stability of Asai local gamma factor for $\GL_2$}
\label{sec-GL(2)}

Before computing integrals, we fix our choice of Haar measures. If ${\rm H}$ is any of the unimodular groups $\G(F), {\rm A}(F), {\rm Z}(F), \U(F),\overline{\U}(F) \mbox{ or } {\rm K}$, we normalize the Haar measure on ${\rm H}$ so that ${\rm vol}({\rm H}\cap {\rm K})=1$. Since $\U(F), \overline{\U}(F)\simeq F$ and ${\rm A}(F)\simeq F^{\times}$, we may identify the measure on $\U(F)$ and $\overline{\U}(F)$ with the additive measure $dx$ on $F$, and the measure on ${\rm A}(F)$ with the multiplicative measure $d^{\times}a$ on $F^{\times}$. (Likewise, for ${\rm Z}(F)\simeq F^{\times}$.) Our normalization is such that $d^{\times}a$ and $dx$ assigns $\o_F^{\times}$ and $\o_F$, respectively, unit measure. If we use the coordinates $g=u(x)zt(a,1)k$, the Haar measure $dg$ on $\G(F)$ can then be decomposed as
\[
dg=\delta_{\rm B}^{-1}(a)dxd^{\times}zd^{\times}adk,
\]
where $\delta_{\B}$ is the modulus character. 
If $C_c^{\infty}(\G(F))$ denotes the space of smooth functions of compact support on $\G(F)$, then for $f\in C_c^{\infty}(\G(F))$, there is a constant $c=c(\B)$ that only depends on $\B$, so that we have the following integration formula \cite{waldspurger} (also see  \cite[Proposition 8.45]{Knapp})
\begin{equation}
\label{int}
\int\limits_{\G(F)}f(g)dg=c \int\limits_{\U(F)\times \T(F)\times\overline{\U}(F)} f(ut\overline{u})\delta_{\B}(t)^{-1}d\overline{u}dtdu.
\end{equation}
We will use this formula in our calculations. (Of course, one can also normalize the Haar measures so that $c=1$ in the above formula.)

\par

A (smooth) character $\eta$ of $F^{\times}$ is said to be unramified if it is trivial on $\o_F^{\times}$. If $\eta$ is unramified, we set its conductor to be $0$; otherwise, the conductor of $\eta$ is the smallest positive integer $d$ for which $\eta$ is trivial upon restriction to $1+\p_F^{d}$. We write $n(\eta)$ to denote the conductor of $\eta$. For $a\in F$, recall the Gauss sum 
\[
  \tau(\eta,\psi,a)=\int_{\o_F^{\times}} \eta^{-1}(u) \psi^{-1}(au) \; d^{\times} u
\]
of $\eta$ relative to $\psi$. We have the following (cf. \cite[23.6 Remark]{BuHe}, for example):
\begin{lemma}
\label{GaussSum}
For $\eta$ ramified, 
\[
   \int_{\o_F^{\times}} \eta^{-1}(u) \psi^{-1}(\varpi_F^ku) \; d^{\times} u
     \begin{cases}
    =  0   & \quad \text{if}\quad k \neq -n(\eta)\\
    \neq 0  & \quad \text{if}\quad k=-n(\eta).
  \end{cases}
    \]
\end{lemma}

For a set $A \subset F$, we denote by $\mathbb{1}_A$ the characteristic function of $A$. For integers $i\geq 0$, $j>0$, put 
\[
  \Phi_{i,j}(x,y)=\mathbb{1}_{\mathfrak{p}_F^i}(x)  \mathbb{1}_{1+\mathfrak{p}_F^j}(y).
\]
We retain the notation of Subsection \ref{sec-asai}, recall that $\Pi$ is an irreducible representation of $\GL_2(E\otimes F)$ and we write $\Pi=\pi\otimes\sigma$ in Case (1). For $m \geq 1$, let $\mathrm{K}_m \leq \GL_2(E \otimes F)$ be $m$-th congruence subgroup defined by
\[
\mathrm{K}_m=
\begin{cases}
\begin{pmatrix} 1+\mathfrak{p}^m_F & \mathfrak{p}^m_F \\ \mathfrak{p}^m_F & 1+\mathfrak{p}^m_F \end{pmatrix}
\times \begin{pmatrix} 1+\mathfrak{p}^m_F & \mathfrak{p}^m_F \\ \mathfrak{p}^m_F & 1+\mathfrak{p}^m_F \end{pmatrix}  & \quad \text{Case (1)},\\
\begin{pmatrix} 1+\mathfrak{p}^m_E & \mathfrak{p}^m_E \\ \mathfrak{p}^m_E & 1+\mathfrak{p}^m_E \end{pmatrix}
 & \quad \text{Case (2)}.\\
\end{cases}
\]
We then have the corresponding subgroups $\J_m^{\mathscr u}$ and $\J_m^{\mathscr l}$ of $\GL_2(E\otimes F)$. 
Choose a Whittaker function $W_{\Pi} \in \mathcal{W}(\Pi,\psi_E)$ satisfying $W_{\Pi}({\rm I})=1$ and let $l$ be a positive integer such that 
$\rho(\mathrm{K}_l)W_{\Pi}=W_{\Pi}$. For $m\geq l$, we form the corresponding Howe Whittaker function $(W_{\Pi})_m=W_{\Pi,m}\in \W(\Pi,\psi_E)$. In Case (1), if $W_{\Pi}=WW', W\in \W(\pi,\psi), W'\in \W(\sigma, \psi^{-1})$, then 
\[
W_{\Pi,m}(g_1,g_2)=W_m(g_1)W_m'(g_2).
\]

\begin{prop}
\label{lhs}
Let $n=n(\omega\chi^2)$ and choose $m \geq \max \{l ,2n(\chi)\}$. Then for $i\geq 3m$,
\[
 Z(s,W_{\Pi,m},\Phi_{i,n},\chi)=
 \begin{cases}
 cq^{-i-n-m}  & \quad \text{if}\; E=F \times F \\
 cq^{-i-n} {\rm vol}((1+\p_E^m) \cap F^{\times})  & \quad \text{if}\; \text{E is a field}.\\
 \end{cases}
\]
\end{prop}

\begin{proof}
Put $\Phi=\Phi_{i,n}$. According to \eqref{int}, we have 
\begin{equation}
\label{UTU decomposition}
\begin{split}
  &Z(s,W_{\Pi,m},\Phi,\chi) \\
  &=c\int\limits_{F^{\times}} \int\limits_{F} W_{\Pi,m}(t(a,1)\overline{u}(x)) \chi(a) |a|_F^{s-1} \left( \int\limits_{F^{\times}} \Phi(e_2t(a,1)\overline{u}(x)z) \omega(z) \chi^2(z) |z|_F^{2s} d^{\times}z \right) dx d^{\times}a\\
  &=c \int\limits_{F^{\times}} \int\limits_{F} W_{\Pi,m}(t(a,1)\overline{u}(x)) \chi(a) |a|_F^{s-1} \left( \int\limits_{F^{\times}} \Phi(e_2\overline{u}(x)z) \omega(z) \chi^2(z) |z|_F^{2s} d^{\times}z \right) dx d^{\times}a.
\end{split}
\end{equation}
The inner integral reduces to
\[
\begin{split}
 \int\limits_{F^{\times}} \Phi(e_2\overline{u}(x)z) \omega(z) \chi^2(z) |z|_F^{2s} d^{\times}z
 =\int\limits_{F^{\times}}\mathbb{1}_{\mathfrak{p}_F^i}(xz)\mathbb{1}_{1+\mathfrak{p}_F^n}(z)  d^{\times}z=
\begin{cases}
    q^{-n}      & \; \text{if } x \in \mathfrak{p}_F^i,\\
    0 & \; \text{otherwise.} 
  \end{cases}
\end{split}
\]
Thus 
\[
 Z(s,W_{\Pi,m},\Phi_{i,n},\chi)=c q^{-n}\int\limits_{F^{\times}} \int\limits_{ \mathfrak{p}_F^i} W_{\Pi,m}(t(a,1)\overline{u}(x)) \chi(a) |a|_F^{s-1} dx d^{\times}a.
\]
Since $i \geq 3m$, $x \in \p_F^i \subset \p_F^{3m}\implies \overline{u}(x) \in \mathrm{J}_m$. But, according to Lemma~\ref{Howe}, the function $W_{\Pi,m}(\cdot)$ is right invariant under $\overline{u}(x)$ for $x\in \p_F^{i}$.  Therefore
\[
\begin{split}
  &Z(s,W_{\Pi,m},\Phi,\chi)=\\
 & c q^{-i-n} \int_{F^{\times}} W_{\Pi,m}(t(a,1))  \chi(a) |a|_F^{s-1} d^{\times} a
 =
 \begin{cases}
  \displaystyle c q^{-i-n} \int_{(1+\mathfrak{p}_F^m)}  \chi(a) |a|_F^{s-1} d^{\times} a  &  \text{Case (1)}\\
   \displaystyle c q^{-i-n} \int_{(1+\mathfrak{p}_E^m) \cap F^{\times} }  \chi(a) |a|_F^{s-1} d^{\times} a  &  \text{Case (2)}. \\
 \end{cases}
\end{split}
\]
Here, the second equality follows from Lemma~\ref{a-support}. 
Now, since $m \geq 2n(\chi)$ and $((1+\p_E^m) \cap F^{\times}) \subset (1+\p_F^{\left[ \frac{m}{2} \right] })$, we obtain the formula.
\end{proof}

By our choice of the measure on $F$, for the characteristic function $\mathbb{1}_{\mathfrak{p}_F^i}$, we have 
\[
    \widehat{\mathbb{1}}_{\mathfrak{p}_F^i}(y)=q^{-i} \mathbb{1}_{\mathfrak{p}_F^{-i}}(y),
\]
and 
\[
\begin{split}
  \widehat{\mathbb{1}}_{1+\mathfrak{p}_F^{j}}(y)
  =\int\limits_F \mathbb{1}_{1+\mathfrak{p}_F^j}(x) \psi(xy)\;dx
  =\psi(-y) \int\limits_F \mathbb{1}_{\mathfrak{p}_F^j}(x) \psi(xy) \; dx 
  =q^{-j} \psi(-y) \mathbb{1}_{\mathfrak{p}_F^{-j}}(y).
\end{split}
\]
Hence 
\begin{equation}
\label{FourierPhi}
  \widehat{\Phi}_{i,j}(x,y)=q^{-i-j} \psi(-y)  \mathbb{1}_{\mathfrak{p}_F^{-j}}(y) \mathbb{1}_{\mathfrak{p}_F^{-i}}(x). 
\end{equation}

\par
Next, we consider the dual side. Suppose $\Pi_1$ and $\Pi_2$ are irreducible admissible representations of $\GL_2(E \otimes F)$ having same central character when restricted to $F^{\times}$. For $j=1,2$, fix $W_{\Pi_j} \in \mathcal{W}(\Pi_j,\psi_E)$ as above, and form the corresponding Howe Whittaker functions $W_{\Pi_1,m}$ and $W_{\Pi_2,m}$, respectively.

\begin{lemma}
\label{rhs}
Choose $m \geq \max \{l,n,2n(\chi) \}$ with $n=n(\omega \chi^2) > 0$. Keeping notations as in Proposition~\ref{lhs}, for $i \geq 3m$, we have
\begin{equation}
\label{compactform}
\begin{split}
 &Z(1-s,\widetilde{W}_{\Pi_1,m}, \widehat{\Phi}_{i,n},\chi^{-1})-Z(1-s,\widetilde{W}_{\Pi_2,m}, \widehat{\Phi}_{i,n},\chi^{-1}) \\
 & =c_{\omega,\chi,\psi} \int\limits_{F^{\times}} [B_m(t(a,1)w,W_{\Pi_1})-B_m(t(a,1)w,W_{\Pi_2})] (\omega\chi)^{-1}(a) |a|_F^{-s} 
   d^{\times}a, \\
 \end{split}
\end{equation}
 where
 \[
 c_{\omega,\chi,\psi}=
 \begin{cases}
 c q^{-i-n+m+2n(1-s)} (\omega\chi^{2})(\varpi_F^{n}) \tau(\omega\chi^2,\psi,\varpi_F^{-n}) &  \text{Case (1)} \\
 c q^{-i-n+2n(1-s)} (\omega\chi^{2})(\varpi_F^{n}) \tau(\omega\chi^2,\psi,\varpi_F^{-n})\mathrm{vol}(\mathfrak{p}_E^{-m} \cap F,\;dx) &  \text{Case (2)}.
 \end{cases}
 \]
\end{lemma}

\begin{proof}
For $j=1,2$, using (\ref{int}) again, we see that $Z(1-s,\widetilde{W}_{\Pi_j,m}, \widehat{\Phi}_{i,n},\chi^{-1})$ equals 
\begin{equation}\label{rhs1}
c \int\limits_{F^{\times}} \int\limits_{F} \widetilde{W}_{\Pi_j,m}(t(a,1)\overline{u}(x)) \chi^{-1}(a) |a|_F^{-s} \left( \int\limits_{F^{\times}} \widehat{\Phi}_{i,n}(e_2t(a,1)\overline{u}(x)z) \omega^{-1}(z) \chi^{-2}(z) |z|_F^{2-2s} d^{\times}z \right) dx d^{\times}a.
\end{equation}
Substituting for the Fourier transform \eqref{FourierPhi}, the inner integral becomes
\[
\begin{split}
   G_{i,n}(x)=
   &\begin{cases}
   q^{-i-n}  \displaystyle  \int_{|z|_F \leq q^{n}} \psi(-z)   \omega^{-1}(z) \chi^{-2}(z)|z|_F^{2(1-s)} d^{\times} z   & \quad \text{if}\quad  |x|_F    \leq q^{i-n}, \\
      q^{-i-n} \displaystyle   \int_{|z|_F \leq \frac{q^i}{|x|_F}} \psi(-z)  \omega^{-1}(z) \chi^{-2}(z)|z|_F^{2(1-s)} d^{\times} z          & \quad \text{if}\quad |x|_F > q^{i-n}.
  \end{cases}
\end{split}
\]
We may re-write the first integral in the above expression as a sum over {\it shells} to obtain 
\[
\begin{split}
q^{-i-n}  \displaystyle  \int_{|z|_F \leq q^{n}} \psi(-z)   (\omega\chi^{2})(z^{-1})|z|_F^{2(1-s)} d^{\times} z
&=q^{-i-n} \sum_{k=-n}^{\infty} \int_{\mathfrak{p}_F^k-\mathfrak{p}_F^{k+1}} \psi(-z) (\omega\chi^{2})(z^{-1})(q^{-k})^{2(1-s)}  d^{\times}z\\
&=q^{-i-n} \sum_{k=-n}^{\infty} q^{-2k(1-s)} (\omega\chi^{2})(\varpi_F^{-k}) \tau(\omega\chi^2,\psi_F,\varpi_{F}^{k}).
\end{split}
\]
Since $\omega\chi^2$ is ramified character of $F^{\times}$ with conductor $n$, by Lemma \ref{GaussSum}, we obtain
\begin{equation}
\label{Gfunction}
 G_{i,n}(x)=q^{-i-n+2n(1-s)} (\omega\chi^{2})(\varpi_F^{n}) \tau(\omega\chi^2,\psi_F,\varpi_F^{-n}) \quad \text{if} \quad |x|_F    \leq q^{i-n}.
\end{equation}
Insert $\widetilde{W}_{\Pi_j,m}(t(a,1)\overline{u}(x))=W_{\Pi_j,m}(t(1,a^{-1})wu(-x))$ in (\ref{rhs1}). It follows from Proposition~\ref{Big Cell-GL(2)} that
\[
 B_m(t(1,a^{-1})wu(-x),W_{\Pi_1})-B_m(t(1,a^{-1})wu(-x),W_{\Pi_2})=0
\]
for $x \notin \mathfrak{p}_E^{-m} \cap F$. Hence
\begin{equation}
\label{UTU Dual decomposition}
 \begin{split}
  &Z(1-s,\widetilde{W}_{\Pi_1,m}, \widehat{\Phi}_{i,n},\chi^{-1})-Z(1-s,\widetilde{W}_{\Pi_2,m}, \widehat{\Phi}_{i,n},\chi^{-1}) \\
  &= c\int\limits_{{F^{\times}}} \int [B_m(t(a,1)wu(-x),W_{\Pi_1})-B_m(t(a,1)wu(-x),W_{\Pi_2})] (\omega\chi)^{-1}(a) |a|_F^{-s} G_{i,n}(x) dx d^{\times}a,\\
  \end{split}
\end{equation}
where the $x$-integral is over $\mathfrak{p}_F^{-m}$ in Case (1), and over $\mathfrak{p}_E^{-m} \cap F$ in Case (2). On the other hand, since $i\geq 3m$, we may choose $m \geq n$ so that $x \in \mathfrak{p}_E^{-m} \cap F\implies x \in \mathfrak{p}_F^{-(i-n)}$. By Lemma \ref{Howe}, part \ref{Howe-item2}, both $W_{\Pi_1,m}$ and $W_{\Pi_2,m}$ are invariant under right translation by elements in $\J_m^{\mathscr u}$. Thus collecting \eqref{Gfunction} and \eqref{UTU Dual decomposition}, we get the desired conclusion.
\end{proof}

 Next, we seek to remove the dependence of the integrand in \eqref{compactform} on $m$. A smooth function $\varphi$ on a subtorus $\mathrm{A} \subset \T$ is said to be \textit{uniformly smooth} if there exists a fixed compact open subgroup $\mathrm{A}_0 \subset \mathrm{A}$ such that $\varphi(aa_0)=\varphi(a)$ for all $a_0 \in \mathrm{A}_0$ and $a \in \mathrm{A}$. We do this on a case-by-case basis.

\subsection{Non-split case} $E$ is a field.

\begin{prop} 
\label{usmth}
In the set-up of Lemma~\ref{rhs}, put $W_i=W_{\Pi_i}$ and $W_{i,m}=(W_i)_m$, $i=1,2$. 
\begin{enumerate}[label=$(\arabic*)$]
\item  There exists a suitably large integer $m\gg 0$ such that, for $i\geq 3m$, we have 
\[
\begin{split}
  &Z(1-s,\widetilde{W}_{1,m}, \widehat{\Phi}_{i,n},\chi^{-1})-Z(1-s,\widetilde{W}_{2,m}, \widehat{\Phi}_{i,n},\chi^{-1})  \\
 &=cq^{-i-n+2n(1-s)} (\omega\chi^{2})(\varpi_F^{n}) \tau(\omega\chi^2,\psi,\varpi_F^{-n})\mathrm{vol}(\mathfrak{p}_E^{-m} \cap F,\;dx)\frac{\mathrm{vol}(\J^{\mathscr u}_{3l})}{\mathrm{vol}(\J^{\mathscr u}_m)} \\
 & \phantom{***********************} \int\limits_{F^{\times}} [B_{3l}(t(a,1)w,W_1)-B_{3l}(t(a,1)w,W_2)] (\omega\chi)^{-1}(a) |a|_F^{-s} 
   d^{\times}a.\\
  \end{split}
\]
\item  $($Uniform smoothness$)$ The function 
\[
a \mapsto   [B_{3l}(t(a,1)w,W^1)-B_{3l}(t(a,1)w,W^2)] \omega^{-1}(a) |a|_F^{-s} 
\]
is uniformly smooth.
\end{enumerate}
\end{prop}

\begin{proof}
We use part \ref{Howe-item3} of Lemma~\ref{Howe} with $m\geq 3l$, to obtain
\begin{equation}
\label{depenOFm}
\begin{split}
  &B_m(t(a,1)w,W_{1})-B_m(t(a,1)w,W_{2})\\
  &=\frac{1}{\mathrm{vol}(\J^{\mathscr u}_m)} \int_{\J^{\mathscr u}_m} [B_{3l}(t(a,1)wu,W_1)-B_{3l}(t(a,1)wu,W_2)] \psi_E^{-1}(u)\;du\\
  &=\frac{1}{\mathrm{vol}(\J^{\mathscr u}_m)} \int_{\J^{\mathscr u}_{3l}} [B_{3l}(t(a,1)wu,W_1)-B_{3l}(t(a,1)wu,W_2)] \psi_E^{-1}(u)\;du\\
  &=\frac{\mathrm{vol}(\J^{\mathscr u}_{3l})}{\mathrm{vol}(\J^{\mathscr u}_m)}  [B_{3l}(t(a,1)w,W_1)-B_{3l}(t(a,1)w,W_2)]. \\
  \end{split}
\end{equation}
Here, we have used Proposition \ref{Big Cell-GL(2)} in deducing the second equality. Plugging \eqref{depenOFm} into \eqref{compactform}, we arrive at the Proposition.
\par
For the second assertion, we note that the function 
\[
a\mapsto B_m(t(a,1)w,W_{1})-B_m(t(a,1)w,W_{2})
\]
is right invariant under $1+\p_F^{3l}$ and is consequently uniformly smooth relative to $m$. 
\end{proof}

 Collecting Proposition~\ref{lhs}, Proposition~\ref{usmth} and  \eqref{lfe}, we obtain 
\[
\begin{split}
&\gamma_{\rm As}(s,\Pi_1,\chi,\psi)-\gamma_{\rm As}(s,\Pi_2,\chi,\psi)\\&=v_m d_{\omega,\chi,\psi}\int\limits_{F^{\times}} [B_{3l}(t(a,1)w,W_1)-B_{3l}(t(a,1)w,W_2)]  (\omega\chi)^{-1}(a) |a|_F^{-s}    d^{\times}a,\\
   \end{split}
\]
where $v_m$ is the volume factor $v_m=\dfrac{\mathrm{vol}(\mathfrak{p}_E^{-m} \cap F)\mathrm{vol}(\J^{\mathscr u}_{3l})}{ \mathrm{vol}(1+\mathfrak{p}_E^m \cap F^{\times})\mathrm{vol}(\J^{\mathscr u}_m)}$ and 
\[
d_{\omega,\chi,\psi}=q^{2n(1-s)} (\omega\chi^{2})(\varpi_F^{n}) \tau(\omega\chi^2,\psi,\varpi_F^{-n})
\]
is a non-zero constant. But the integral on the right hand side is $0$ for a suitable highly ramified $\chi$. For instance, if $\chi$ is such that $n(\chi)>3l$, then the integral vanishes thus giving us stability of Asai $\gamma$-factor under highly ramified twists.

To prove the corresponding statement for $L$- and $\varepsilon$-factors, we only need to prove stability for the $L$-function. With \cite[Corollary 4.3]{Matringe} in hand concerning the poles of the Asai $L$-function, this can be proved exactly as in Jacquet and Shalika \cite[Proposition 5.1]{JS}.

\subsection{Split case} $E=F\times F$.

Here $\Pi_1=\pi_1 \otimes \sigma$ and $\Pi_2=\pi_2 \otimes \sigma$ with $\pi_1$ and $\pi_2$ having the same central character. Also, $\omega=\omega_{\pi_1}\omega_{\sigma}=\omega_{\pi_2}\omega_{\sigma}$ in this case. Suppose $W_{\Pi_1,m}=W_{1,m} W^{\prime}_m$ and $W_{\Pi_2,m}=W_{2,m}  W^{\prime}_m$, where $W_{j,m} \in \mathcal{W}(\pi_j,\psi)$, $j=1,2$, and $W^{\prime}_m \in \mathcal{W}(\sigma,\psi^{-1})$.
\begin{prop} 
\label{usmth-RS}
Let $\omega,\chi,n,l$ and rest of the notation be as in Lemma~\ref{rhs}. 
\begin{enumerate}[label=$(\arabic*)$]
\item  There exists a suitably large integer $m\gg 0$ such that, for $i\geq 3m$, we have 
\[
\begin{split}
  &Z(1-s,\widetilde{W}_{1,m},\widetilde{W}^{\prime}_m, \widehat{\Phi}_{i,n},\chi^{-1})-Z(1-s,\widetilde{W}_{2,m},\widetilde{W}^{\prime}_m, \widehat{\Phi}_{i,n},\chi^{-1})  \\
 &=cq^{-i-n+m+2n(1-s)} (\omega\chi^{2})(\varpi_F^{n}) \tau(\omega\chi^2,\psi,\varpi_F^{-n})\frac{\mathrm{vol}(\J^{\mathscr u}_{3l})}{[\mathrm{vol}(\J^{\mathscr u}_m)]^2} \\
 & \phantom{*******************} \int_{\mathfrak{p}_F^{-9l}} [B_{3l}(t(a,1)w,W_1)-B_{3l}(t(a,1)w,W_2)] j_{\sigma}(a) (\omega\chi)^{-1}(a) |a|_F^{-s} 
   d^{\times}a.\\
  \end{split}
\]
The integral can run over $F^{\times}$. 
\item  $($Uniform smoothness$)$ The function 
\[
  a \mapsto [W^1_{3l}(t(a,1)w)-W^2_{3l}(t(a,1)w) ] j_{\sigma}(a)\omega^{-1}(a) |a|_F^{-s} 
\]
is uniformly smooth for $a \in \mathfrak{p}_F^{-9l}$.
\end{enumerate}

\end{prop}

\begin{proof}
As in \eqref{depenOFm}, we obtain
\[
\begin{split}
& \int\limits_{F^{\times}} [B_m(t(a,1)w,W_{\Pi_1})-B_m(t(a,1)w,W_{\Pi_2})] (\omega\chi)^{-1}(a) |a|_F^{-s} 
   d^{\times}a\\
 & =\int\limits_{F^{\times}} [(W_1)_{m}(t(a,1)w)-(W_2)_{m}(t(a,1)w) ]W_m^{\prime}(t(a,1)w) (\omega\chi)^{-1}(a) |a|_F^{-s} d^{\times} a \\
  &=\frac{\mathrm{vol}(\J^{\mathscr u}_{3l})}{\mathrm{vol}(\J^{\mathscr u}_m)} \int\limits_{F^{\times}}  [(W_1)_{3l}(t(a,1)w)-(W_2)_{3l}(t(a,1)w) ]W_m^{\prime}(t(a,1)w) 
  (\omega\chi)^{-1}(a) |a|_F^{-s} d^{\times} a\\
\end{split}
\]
for $m \geq 3l$. By Lemma~\ref{a-support} \ref{a-support-item2}, we get
\[
 \begin{split}
& \int_{F^{\times}} [B_m(t(a,1)w,W_{\Pi_1})-B_m(t(a,1)w,W_{\Pi_2})] (\omega\chi)^{-1}(a) |a|_F^{-s} 
   d^{\times}a\\
   &=\frac{\mathrm{vol}(\J^{\mathscr u}_{3l})}{\mathrm{vol}(\J^{\mathscr u}_m)} \int_{\mathfrak{p}_F^{-9l}}  [(W_{1})_{3l}(t(a,1)w)-(W_2)_{3l}(t(a,1)w) ]W_m^{\prime}(t(a,1)w) 
  (\omega\chi)^{-1}(a) |a|_F^{-s} d^{\times} a.\\
 \end{split}
\]
Choosing $m=3n$ with $n \geq 2l$, we apply Proposition~ \ref{soudry} \ref{soudry-item2} to obtain 
\[
  \mathrm{vol}(\J_m^{\mathscr u}) W_m^{\prime}(t(a,1)w)=j_{\sigma}(a) \;\; \text{for} \;\; |a|_F \leq q^{2m}.
\]
Then inserting this into the above expression we get 
\begin{equation}
\label{a-integral}
 \begin{split}
& \int_{F^{\times}} [B_m(t(a,1)w,W_{\Pi_1})-B_m(t(a,1)w,W_{\Pi_2})] (\omega\chi)^{-1}(a) |a|_F^{-s} 
   d^{\times}a\\
   &=\frac{\mathrm{vol}(\J^{\mathscr u}_{3l})}{[\mathrm{vol}(\J^{\mathscr u}_m)]^2} \int_{\mathfrak{p}_F^{-9l}}  [(W_1)_{3l}(t(a,1)w)-(W_2)_{3l}(t(a,1)w) ] j_{\sigma}(a)
  (\omega\chi)^{-1}(a) |a|_F^{-s} d^{\times} a.\\
 \end{split}
\end{equation}
Plugging \eqref{a-integral} into \eqref{compactform}, we obtain the desired formula.
\par
For the second assertion, it follows from Proposition \ref{soudry} \ref{soudry-item3} that for $a\in \p_F^{-9l}$, the function is right invariant under $1+\p_F^{6l}$. 
\end{proof}

Putting Proposition~\ref{lhs}, Proposition~\ref{usmth-RS} and \eqref{lfe} together, we get 
\[
\begin{split}
 &\gamma(s,(\pi_1\otimes\chi),\psi)-\gamma(s,(\pi_2\otimes\chi),\psi)\\&=v_m d_{\omega,\chi,\psi}\int\limits_{\mathfrak{p}_F^{-9l}} [B_{3l}(t(a,1)w,W_1)-B_{3l}(t(a,1)w,W_2)]  j_{\sigma}(a)
  (\omega\chi)^{-1}(a) |a|_F^{-s} d^{\times} a,\\
   \end{split}
\]
where $v_m=\dfrac{\mathrm{vol}(\J^{\mathscr u}_{3l})}{ [\mathrm{vol}(\J^{\mathscr u}_m)]^2}$ and $
d_{\omega,\chi,\psi}=q^{2m+2n(1-s)} (\omega\chi^{2})(\varpi_F^{n}) \tau(\omega\chi^2,\psi,\varpi_F^{-n})\neq 0$. By uniform smoothness, the integral on the right hand side is $0$ for $\chi$ satisfying $n(\chi)>6l$. Since $L(s,(\pi\otimes\chi)\times\sigma)=1$ for a highly ramified $\chi$ according to \cite[Proposition 5.1]{JS}, this concludes the proof of Theorem~\ref{main} in the split case.

\begin{acknowledgment}
We thank Roger Howe for sending us a copy of \cite{H}.
\end{acknowledgment}

\end{document}